\documentclass{amsart}
\usepackage[utf8]{inputenc}
\usepackage[english]{babel}							
\usepackage[T1]{fontenc}\usepackage{amsthm}
\usepackage{amssymb}
\usepackage{amsmath}
\usepackage{graphicx} 
\usepackage{tikz}
\usepackage{ytableau,tikz,varwidth}
\usetikzlibrary{calc}
\usepackage{tikz-cd}

\tikzset{
  curarrow/.style={
  rounded corners=8pt,
  execute at begin to={every node/.style={fill=red}},
    to path={-- ([xshift=-50pt]\tikztostart.center)
    |- (#1) node[fill=white] {$\scriptstyle \delta$}
    -| ([xshift=50pt]\tikztotarget.center)
    -- (\tikztotarget)}
    }
}

\newtheorem{theorem}{Theorem}[section]
\newtheorem{corollary}[theorem]{Corollary}
\newtheorem{lemma}[theorem]{Lemma}
\theoremstyle{definition}

\newtheorem{example}[theorem]{Example}

\newtheorem{conj}[theorem]{Conjecture}

\newtheorem*{ack}{Acknowledgement}
\DeclareMathOperator{\Ima}{Im}

\title{Super edge-magic total strength of some unicyclic graphs}
\author{Nayana Shibu Deepthi}
\address{Department of Pure and Applied Mathematics, Graduate School of Information Science and Technology, Osaka University, Suita, Osaka 565-0871, Japan}
\email{nayanasd@ist.osaka-u.ac.jp}

\begin{document}

\begin{abstract}
Let $G$ be a finite simple undirected $(p,q)$-graph, with vertex set $V(G)$ and edge set $E(G)$ such that $p=|V(G)|$ and $q=|E(G)|$. A super edge-magic total labeling $f$ of $G$ is a bijection $f\colon V(G)\cup E(G)\longrightarrow \{1,2,\dots , p+q\}$ such that for all edges $u v\in E(G)$, $f(u)+f(v)+f(u v)=c(f)$, where $c(f)$ is called a magic constant, and $f(V(G))=\{1,\dots , p\}$. The minimum of all $c(f)$, where the minimum is taken over all the super edge-magic total labelings $f$ of $G$, is defined to be the super edge-magic total strength of the graph $G$. In this article, we work on certain classes of unicyclic graphs and provide shreds of evidence to conjecture that the super edge-magic total strength of a certain family of unicyclic $(p,q)$-graphs is equal to $2q+\frac{n+3}{2}$.
\end{abstract}

\maketitle



\section{Introduction}\label{sec:2}

Throughout this article, we only consider finite simple undirected graphs. Let $G$ be a finite simple undirected graph with vertex set $V(G)$ and edge set $E(G)$. Let us assume that $|V(G)|=p$ and $|E(G)|=q$, then $G$ is called a {\em $(p,q)$-graph}.
For any undefined terms and notations in this article, we follow \cite{Bon}.

A {\em magic valuation} of a graph $G$ was first described by Kotzig and Rosa \cite{MagV} as a bijection $f\colon V(G)\cup E(G)\longrightarrow \{1,2,\dots , p+q\}$ such that for all edges $u v\in E(G)$, the sum $f(u)+f(v)+f(uv)$ is a constant, called as the {\em magic constant} of $f$. 
A graph is said to be {\em magic} if it has a magic valuation.
Ringel and Llado \cite{RinL} rediscovered this idea and gave it the name {\em edge-magic}. 
We shall use the phrase {\em edge-magic total}, as developed by Wallis \cite{Wa} to distinguish this usage from that of other forms of labelings that utilize the word magic. 

In \cite{MagV}, Kotzig and Rosa have demonstrated that the cycles $C_{n}$, $n\geq 3$ and the complete bipartite graphs $K_{m,n}$ with $m,n\geq 1$, are edge-magic total graphs. They also established that a complete graph $K_{n}$ is an edge-magic total graph if and only if $1\leq n\leq 6$, and further proved that the disjoint union of $n$ copies of $P_{2}$ has an edge-magic total labeling if and only if $n$ is odd. The bibliography section includes references to several pieces of pertinent literature, including \cite{BR, Gal, NgB, RinL, SB}.

Avadayappan, Vasuki, and Jeyanthi \cite{MagicS} introduced the concept of {\em edge-magic total strength} of a graph $G$ as the smallest magic constant over all edge-magic total labelings of $G$. Let us denote it by $em(G)$. Let $f$ be an edge-magic total labeling of $G$ with magic constant $c(f)$. By definition, 
$$em(G)=\min \big\{c(f)\colon f \textrm{ be an edge-magic total labeling of } G \big\}.$$

In this study, we explore an edge-magic total labeling $f$ of the $(p,q)$-graph $G$, such that $f(V(G))=\{1,2,\dots ,p\}$. Enomoto, Llado, Nakamigawa, and Ringel \cite{SEMg} has defined this type of labeling as a {\em super edge-magic total labeling}. If a graph $G$ contains a super edge-magic total labeling, then $G$ is called a {\em super edge-magic total}.

Enomoto, Llado, Nakamigawa, and Ringel \cite{SEMg} established that a complete graph $K_{n}$ is super edge-magic total if and only if $n = 1, 2,$ or $3$ and also proved that a complete bipartite graph $K_{m,n}$ is super edge-magic total if and only if $m = 1$ or $n = 1$. As demonstrated in \cite{SEMg}, cycles $C_{n}$ are super edge-magic total if and only if $n$ is odd. Some further results on the super edge-magic total graph can be found in \cite{Coro, FIM, Gal}.

A necessary and sufficient condition for a graph to be a super edge-magic total is stated in the following lemma. 

\begin{lemma}[{\cite[Lemma 1]{FIMun}}]\label{lem:SEM}
A $(p,q)$-graph $G$ is a super edge-magic total graph if and only if there exists a bijection $f\colon V(G) \longrightarrow \{1,2,\dots , p\}$, such that the set $\{f(u)+f(v)\colon u v\in E(G)\}$ is consecutive. And, $f$ extends to a super edge-magic total labeling with magic constant  $c(f)=p+q+\min\{f(u)+f(v)\colon u v\in E(G)\}$.
\end{lemma}

In this case, the vertex labeling $f$ can be extended to a super edge-magic total labeling of $G$ by defining $f(uv) = p + q + \min\big\{f(u) + f(v) \colon uv \in E(G)\big\} - f(u) - f(v)$, for every edge $uv \in E(G)$.

For any regular super edge-magic total graph, we have the following result.

\begin{lemma}[{\cite[Lemma 4]{FIMun}}]\label{lem:cyc}
Let $G$ be an $r$-regular $(p, q)$-graph, where $r > 0$. Let $f$ be any super edge-magic total labeling of $G$. Then $q$ is odd and $c(f)=\frac{4p + q + 3}{2}$, for all super edge-magic total labeling $f$.
\end{lemma}

By using Lemma \ref{lem:cyc}, we can derive that for an odd cycle $C_{n}$ and with any super edge-magic total labeling  $f$  of $C_{n}$, 
\begin{equation*}
\begin{split}
c(f) & =  2n+ \min\{f(u)+f(v)\colon u v\in E(C_{n})\} = 2n+ \frac{n+3}{2}\\
& \implies \min\{f(u)+f(v)\colon u v\in E(C_{n})\} =  \frac{n+3}{2}.
\end{split}
\end{equation*}

The article \cite{newSEM} demonstrates that by considering a super edge-magic total labeling of a super edge-magic total graph, we can add vertices and edges to the given graph such that the new graph constructed is also super edge-magic total. 

\begin{theorem}[{\cite[Theorem 2.4]{newSEM}}]\label{thm:new}
Let $G_{p}$ be a connected super edge-magic total $(p,q)$-graph with $p\geq 3$. Let $f$ be a super edge-magic total labeling  of $G_{p}$ and let us consider $F_{f}(G_{p})=\{f(u)+f(v)\colon u v\in E(G_{p})\}$. Let $\max (F_{f}(G_{p}))= p+t$, and for some $a\in V(G_{p})$, $f(a)=t$. We construct a graph $\widetilde{G_{p}}$, by taking each copy of $G_{p}$ and $m K_{1},\ m\geq 1$ and connecting all the vertices of $mK_{1}$ to the vertex $a\in V(G_{p})$. Then, $\widetilde{G_{p}}$ is also a super edge-magic total graph.
\end{theorem}

In \cite{SEMS}, the concept of {\em super edge-magic total strength} of a graph $G$ was introduced, where it is defined as the minimum magic constant, the minimum taken over all the super edge-magic total labelings of $G$. We denote it as $sm(G)$. So, we have $$sm(G)=\min \big\{c(f)\colon f \textrm{ is a super edge-magic total labeling of } G\big\}.$$

The super edge-magic total strength of an odd cycle $C_{n}$ was obtained as $\frac{5n+3}{2}$, in \cite{SEMS}. More results regarding the super edge-magic total strength of certain families of graphs are demonstrated in \cite{SEMS, Uni}.

Let $G$ be a super edge-magic total $(p,q)$-graph. For any $v\in V(G)$, the number of edges adjacent to vertex $v$ is called the {\em degree} of $v$, denoted by $\mathrm{deg}(v)$. Let  $f$ be a super edge-magic total labeling of $G$, with the magic constant $c(f)$.  As noted in \cite{MagicS}, each edge's magic constants are added together to produce the following result:

\begin{equation}\label{eq:1}
q c(f)=\sum_{v\in V(G)}\mathrm{deg}(v)f(v)+\sum_{e\in E(G)}f(e).
\end{equation}

Also, since $\Ima(f)= \{1,2,\dots , p+q\}$ and $f(V(G))= \{1,2, \dots ,p\}$, we can derive that $p+q+3\leq sm(G)\leq 3p$. 

In this paper, we are interested in the family of unicyclic graphs which consists of an odd cycle $C_{n}$ and $k_i$ number of pendant vertices adjacent to each ${i}\in V(C_{n})$. 

Let us consider $G(n;k_{1},\dots ,k_{n})$ to be the unicyclic $(p,q)$-graph consisting of an odd cycle $C_n=\{a_1,a_2,\dots ,a_n\}$ and $k_i$ number of pendant vertices adjacent to the vertex $a_i,\ 1\leq i\leq n$.  
Swaminathan and Jeyanthi \cite{Uni} established a range for the super edge-magic total strength of this family of unicyclic graphs. 

\begin{theorem}[{\cite[Theorem 4]{Uni}}]\label{thm:uni}
 The unicyclic $(p,q)$-graph $G(n;k_{1},\dots ,k_{n})$, where $n=2s+1$,  is a super edge-magic total graph and
\begin{equation*}
\begin{split}
& 2q+2+\frac{1}{q}\bigg(m_2+2m_3+\dots + (n-1)m_n+\frac{n(n-1)}{2}\bigg) \leq s m\big(G(n;k_{1},\dots ,k_{n})\big)\\
& \leq 2(k_1+k_3+\dots +k_{2s+1})+3(k_2+k_4+\dots +k_{2s})+2n+s+2,
\end{split}
\end{equation*}
where $m_1\geq m_2\geq\dots\geq m_n$ are integers such that 
$$\{m_1,m_2,\dots ,m_n\}=\{k_1,k_2,\dots , k_n\}.$$
\end{theorem}

\begin{corollary}[{\cite[Corollary 4.1]{Uni}}]\label{lem:uni_coro}
For any unicyclic graph $G(n;k_{1},\dots ,k_{n})$ such that $k_{i}=k,$ for any $1\leq i\leq n$,
$$sm(G(n;k,\dots ,k))=2n(k+1)+\frac{n+3}{2}.$$ 
\end{corollary}

In this article, we will be investigating the super edge-magic total strength of the family of unicyclic graphs having an odd cycle $C_{n}$ and $k_i$ number of pendant vertices adjacent to each ${i}\in V(C_{n})$. We will compute the super edge-magic total strength of certain graphs and provide supporting evidence for the following conjecture.

\begin{conj}\label{cnj:main}
 Let $G(n;k_{1},\dots ,k_{n})$ be the super edge-magic total unicyclic $(p,q)$-graph consisting of an odd cycle $C_n=\{a_1,a_2,\dots ,a_n\}$ and $k_i$ number of pendant vertices adjacent to each $a_i,\ 1\leq i\leq n$. Then,
 $$s m(G(n;k_{1},\dots ,k_{n})) = 2q + \frac{n+3}{2}. $$
\end{conj}

In this article, we particularly examine three specific graphs belonging to the family of unicyclic graphs $G(n;k_{1},\dots ,k_{n})$ and provide substantial evidence in favor of our Conjecture \ref{cnj:main}.

A brief structure of this paper is as follows.  
Section \ref{sec:3} deals with the unicyclic graph $G(n;k_{1},\dots ,k_{n})$ with $k_{i}=k$, for any $1\leq i\leq n-1$ and $k_{n}=k+c$, where $1\leq c < \frac{2n(k+1)}{n-3}$. We further prove that the super edge-magic total strength of this graph satisfies our main conjecture.
In Section \ref{sec:4}, we study $G(n;k,\dots ,k,k-c)$, where $1\leq c\leq k$ and prove that the super edge-magic total strength of this family of graphs satisfies the conjecture. 
Section \ref{sec:5} is about the unicyclic graph $G(n,k_{1},\dots ,k_{n})$ with $k_{i}=k,$ if $i\neq r,n-r$ and $k_{r}=k_{n-r}=k+1$ for any odd number $r$, $1\leq r< n$. Further, we prove that this family of graphs also satisfies our main conjecture.
All of our conclusions from this study are included in Section \ref{sec:6}.

\begin{ack}
I would like to thank Professor Akihiro Higashitani, Department of Pure and Applied Mathematics, Osaka University, for his constant support and valuable comments that greatly improved the manuscript.
\end{ack}


\section{Unicyclic graph $G_{n,k,c}$}\label{sec:3}

Let $G_{n,k,c}:= G(n;k,\dots , k,k+c)$, where $1\leq c < \frac{2n(k+1)}{n-3}$. That is, $G_{n,k,c}$ is the unicyclic graph consisting of an odd cycle $C_{n}=\{a_{1},a_{2},\dots , a_{n}\}$, with $k$ number of pendant vertices adjacent to each of the vertices $a_{i},\ 1\leq i\leq n-1$ and $k+c$ number of pendant vertices adjacent to vertex $a_{n}$. For illustration, see Figure \ref{fig1a}. The number of vertices and edges of the graph $G_{n,k,c}$ is $p=q=n(k+1)+c$. Let the vertex set $V(G_{n,k,c})$ be $$V(C_{n})\cup \big\{a_{i,j}\colon 1\leq i\leq n-1, 1\leq j\leq k\big\}\cup \big\{a_{n,j}\colon 1\leq j\leq k+c\big\}$$ and let the edge set $E(G_{n,k,c})$ be
\begin{equation*}
\begin{split}
& E(C_{n})\cup\big\{a_{i}a_{i,j}\colon 1\leq i\leq n-1,1\leq j\leq k\big\}\cup\big\{a_{n}a_{n,j}\colon 1\leq j \leq k+c\big\}.
\end{split}
\end{equation*}

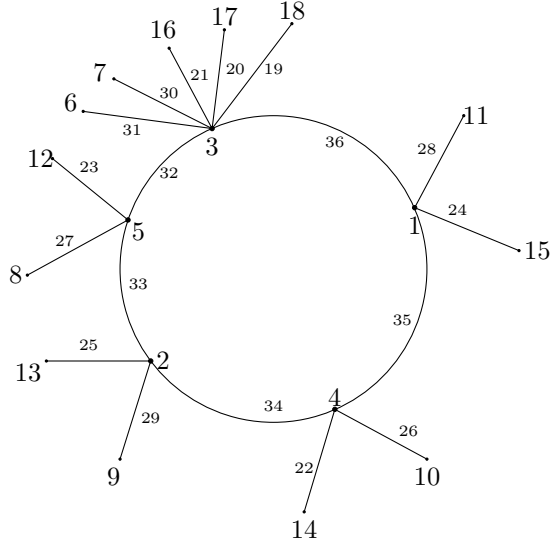
\begin{figure}[ht]
\centering
\begin{tikzpicture}[scale=0.816]
\filldraw[color=black!100, fill=white!5,  thin](6,3) circle (2.5);
\filldraw [black] (5,5.29) circle (1pt);
\filldraw [black] (7,0.712) circle (1pt);
\filldraw [black] (3.63,3.8) circle (1pt);
\filldraw [black] (8.3,4) circle (1pt);
\filldraw [black] (4,1.5) circle (1pt);
\filldraw [black] (8.5,-0.1) circle (0.5pt);
\draw[black, thin] (5,5.29) -- (4.3,6.6);
\filldraw [black] (4.3,6.6) circle (0.5pt);
\draw[black, thin] (5,5.29) -- (5.2,6.9);
\filldraw [black] (5.2,6.9) circle (0.5pt);
\draw[black, thin] (5,5.29) -- (6.3,7);
\draw[black, thin] (5,5.29) -- (2.9,5.57);
\draw[black, thin] (5,5.29) -- (3.4,6.1);
\draw[black, thin] (8.3,4) -- (10,3.3);
\draw[black, thin] (8.3,4) -- (9.1,5.5);
\draw[black, thin] (7,0.712) -- (6.5,-0.96);
\draw[black, thin] (7,0.712) -- (8.5,-0.1);
\draw[black, thin] (4,1.5) -- (3.5,-0.1);
\draw[black, thin] (4,1.5) -- (2.3,1.5);
\draw[black, thin] (3.63,3.8) -- (1.99,2.9);
\draw[black, thin] (3.63,3.8) -- (2.4,4.8);
\filldraw [black] (2.3,1.5) circle (0.5pt);
\filldraw [black] (3.5,-0.1) circle (0.5pt);
\filldraw [black] (6.3,7) circle (0.5pt);
\filldraw [black] (6.5,-0.96) circle (0.5pt);
\filldraw [black] (9.1,5.5) circle (0.5pt);
\filldraw [black] (1.99,2.9) circle (0.5pt);
\filldraw [black] (2.4,4.8) circle (0.5pt);
\filldraw [black] (2.9,5.57) circle (0.5pt);
\filldraw [black] (3.4,6.1) circle (0.5pt);
\filldraw [black] (10,3.3) circle (0.5pt);
\filldraw [black] (5,5.29) node[anchor=north] {$3$};
\filldraw [black] (8.3,4) node[anchor=north] {$1$};
\filldraw [black] (7,1.2) node[anchor=north] {$4$};
\filldraw [black] (4.2,1.8) node[anchor=north] {$2$};
\filldraw [black] (3.8,3.9) node[anchor=north] {$5$};
\filldraw [black] (8.1,2.4) node[anchor=north] {{\tiny  $35$}};
\filldraw [black] (7,5.3) node[anchor=north] {{\tiny $36$}};
\filldraw [black] (4.3,4.8) node[anchor=north] {{\tiny $32$}};
\filldraw [black] (3.8,3) node[anchor=north] {{\tiny $33$}};
\filldraw [black] (6,1) node[anchor=north] {{\tiny $34$}};
\filldraw [black] (2.7,6) node[anchor=north] {$6$};
\filldraw [black] (3.7,5.5)  node[anchor=north] {{\tiny  $31$}};
\filldraw [black] (3.17,6.5) node[anchor=north] {$7$};
\filldraw [black] (4.3,6.1)  node[anchor=north] {{\tiny $30$}};
\filldraw [black] (4.2,7.2)node[anchor=north] {$16$};
\filldraw [black] (4.8,6.4)  node[anchor=north] {{\tiny $21$}};
\filldraw [black] (5.2,7.4) node[anchor=north] {$17$};
\filldraw [black] (5.38,6.5)  node[anchor=north] {{\tiny  $20$}};
\filldraw [black] (6.3,7.5) node[anchor=north] {$18$};
\filldraw [black] (6,6.5)  node[anchor=north] {{\tiny $19$}};
\filldraw [black] (9.3,5.8) node[anchor=north] {$11$};
\filldraw [black] (8.5,5.2)  node[anchor=north] {{\tiny $28$}};
\filldraw [black] (9,4.2) node[anchor=north] {{\tiny $24$}};
\filldraw [black]  (10.3,3.6) node[anchor=north] {$15$};
\filldraw [black] (8.5,-0.1) node[anchor=north] {$10$};
\filldraw [black] (8.2,0.6)  node[anchor=north] {{\tiny $26$}};
\filldraw [black] (6.5,-0.96) node[anchor=north] {$14$};
\filldraw [black] (6.5,0)  node[anchor=north] {{\tiny $22$}};
\filldraw [black] (3.4,-0.1) node[anchor=north] {$9$};
\filldraw [black] (4,0.8)  node[anchor=north] {{\tiny $29$}};
\filldraw [black] (2,1.6) node[anchor=north] {$13$};
\filldraw [black] (2.99,2)  node[anchor=north] {{\tiny $25$}};
\filldraw [black] (1.8,3.2) node[anchor=north] {$8$};
\filldraw [black] (2.6,3.7)  node[anchor=north] {{\tiny $27$}};
\filldraw [black] (2.2,5.1)  node[anchor=north] {$12$};
\filldraw [black] (3,4.9)  node[anchor=north] {{\tiny $23$}};
\end{tikzpicture}
\caption{The graph $G_{5,2,3}$}
\label{fig1a}
\end{figure}

\begin{theorem}\label{thm:k+c}
The unicyclic graph $G_{n,k,c}$ is a super edge-magic total graph with super edge-magic total strength given by $$s m(G_{n,k,c})=2n(k+1)+2c+\frac{n+3}{2}.$$
\end{theorem}

\begin{proof}

By Theorem \ref{thm:uni}, the graph $G_{n,k,c}:= G(n;k,\dots ,k ,k+c)$, $1\leq c < \frac{2n(k+1)}{n-3}$, is super edge-magic total. From \eqref{eq:1}, for any super edge-magic total labeling $f$ of $G_{n,k,c}$, we have   
\begin{equation*}
\begin{split}
q c(f)&=\sum_{v\in V(G_{n,k,c})}\deg(v)f(v)+\sum_{e\in E(G_{n,k,c})}f(e)\\
& = q(2q+1)+ (k+1)\sum_{a_{i}\in V(C_{n})}f(a_{i})+c f(a_{n}) .
\end{split}
\end{equation*}

By assigning smaller labels to vertices with higher degrees, we get the least possible value of $c(f)$ as
\begin{equation*}
\begin{split}
& c(f)\geq 2q+1+\frac{(k+1)n(n+1)}{2q}+\frac{c}{q}.
\end{split}
\end{equation*}

Hence, we have $$sm(G_{n,k,c})\geq 2q+1+\frac{n(k+1)(n+1)}{2q}+\frac{c}{q}.$$

Since $sm(G_{n,k,c})$ is an integer, we consider the integer part of $\bigg( \frac{n(k+1)(n+1)}{2q}+\frac{c}{q}\bigg)$. We have $\frac{n+1}{2}-\bigg( \frac{n(k+1)(n+1)}{2q}+\frac{c}{q}\bigg)=\frac{(n-1)c}{2(n(k+1)+c)}$. Since $1\leq c < \frac{2n(k+1)}{n-3}$, we have $(n-1)c < 2(n(k+1)+c)$. Therefore, $0< \frac{(n-1)c}{2(n(k+1)+c)} < 1$. 
Hence,
\begin{align*}  
s m(G_{n,k,c})&\geq  2q+1+\frac{(n+1)}{2}\\
& = 2n(k+1)+2c+\frac{n+3}{2}.
\end{align*}
That is, 
\begin{equation}\label{eq:lb_k+c}
s m(G_{n,k,c})\geq 2n(k+1)+2c+\frac{n+3}{2}.
\end{equation}

By Theorem \ref{thm:new} and Theorem \ref{thm:uni}, the graph $G_{n,k,c}$ can be seen as the super edge-magic total graph constructed from the graph $G(n;k,\dots , k)$ with the super edge-magic total labeling  $f^{'}$ of $G(n;k,\dots , k)$ defined as follows.

For $1\leq i \leq n,$
\begin{equation*}
\begin{split}
&f^{'}(a_{i})= 
  \begin{cases} 
   \displaystyle{ \frac{i+1}{2}}& \text{ if } i \text{ is odd},  \\ \\
   \displaystyle{\frac{n+i+1}{2}}&  \text{ if } i \text{ is even}.
  \end{cases}\\
& f^{'}(a_{i,j})= n(k+1)-(n-1)(j-1)-(i-1),\ 1\leq i\leq n-1,\ 1\leq j\leq k.\\
& f^{'}(a_{n,k})= n+j,\ 1\leq j\leq k. 
\end{split}
\end{equation*}

Now, we consider a vertex labeling $f\colon V(G_{n,k,c})  \longrightarrow \{1,\dots , p\}$ as follows:

\begin{equation}\label{eq:semGn}
\begin{split}
&f(v)= 
  \begin{cases} 
   \displaystyle{ f^{'}(v)}& \text{ if } v\in V(G(n;k,\dots , k)),  \\ \\
   \displaystyle{n(k+1)+j-k}&  \text{ if } v= a_{n,j}, \text{ for } k+1\leq j \leq k+c.
  \end{cases}
\end{split}
\end{equation}

As per the labeling defined in \eqref{eq:semGn}, for any $u v\in E(G_{n,k,c})$ we observe the following.
\begin{itemize}
\item If $u, v\in V(G_{n,k,c})$, since $f^{'}$ is a super edge-magic total labeling of the graph $G(n;k,\dots , k)$, then $\{f(u)+f(v)\}=\{f^{'}(u)+f^{'}(v)\}$ is a consecutive sequence with highest element $n(k+1)+\frac{n+1}{2}$.
\item If $u=a_{n}$ and $v=a_{n,j}$, for $ k+1\leq j\leq k+c,$ then  $\{f(u)+f(v)\}=\big\{\frac{n+1}{2}+ n(k+1)+1 , \dots , \frac{n+1}{2}+ n(k+1)+ c \big\}$ is a consecutive sequence.
\end{itemize}

Therefore, we observe that $\{f(u)+f(v)\colon u v\in E(G_{n,k,c})\}$ is a consecutive sequence and we have $\min \{f(u)+f(v)\colon u v\in E(G_{n,k,c})\}= \frac{n+3}{2}$. By Lemma \ref{lem:SEM}, the vertex labeling $f$ extends to a super edge-magic total labeling of $G_{n,k,c}$ with $c(f)= 2n(k+1)+2c+\frac{n+3}{2}$. Hence,

\begin{equation}\label{eq:ub_k+c}
s m(G_{n,k,c})\leq 2n(k+1)+2c+\frac{n+3}{2}.
\end{equation}

From \eqref{eq:lb_k+c} and \eqref{eq:ub_k+c}, we have
\begin{equation*}
 s m(G_{n,k,c}) = 2n(k+1)+2c+\frac{n+3}{2}.
\end{equation*}

\end{proof}

\begin{example}
Super edge-magic total labeling of the graph $G_{5,2,3}$ with strength $s m(G_{5,2,3})= 40$, is illustrated in Figure \ref{fig1a}. 
\end{example}

\begin{example}
Super edge-magic total labeling of $G_{9,3,4}$ with $s m(G_{9,3,4})= 86$ is illustrated in Figure \ref{fig1}.
\end{example}
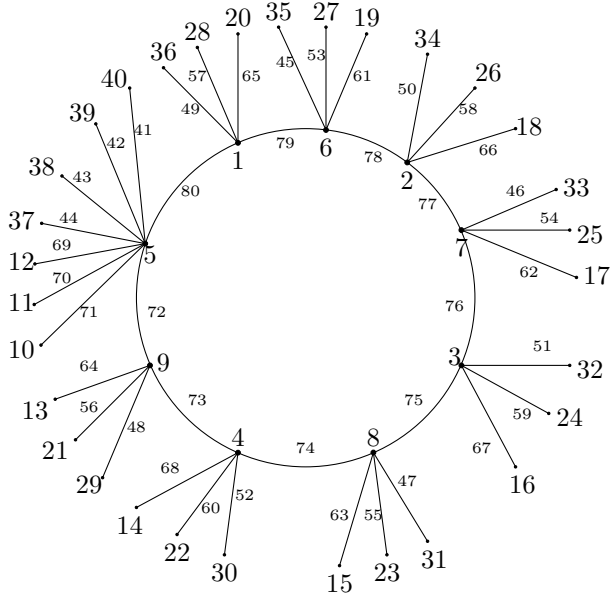
\begin{figure}[ht]
\centering
\begin{tikzpicture}[scale=0.9]
\filldraw[color=black!100, fill=white!5,  thin](6,3) circle (2.5);
\filldraw [black] (3.7,2) circle (1pt);
\filldraw [black] (8.3,2) circle (1pt);
\filldraw [black] (5,5.29) circle (1pt);
\filldraw [black] (5,0.71) circle (1pt);
\filldraw [black] (6.3,5.48) circle (1pt);
\filldraw [black] (7,0.712) circle (1pt);
\filldraw [black] (7.5,5) circle (1pt);
\filldraw [black] (3.63,3.8) circle (1pt);
\filldraw [black] (8.3,4) circle (1pt);
\draw[black, thin] (3.7,2) -- (3,0.34);
\draw[black, thin] (3.7,2) -- (2.6,0.9);
\draw[black, thin] (3.7,2) -- (2.3,1.5);
\draw[black, thin] (8.3,2) -- (9.1,0.5);
\draw[black, thin] (8.3,2) -- (9.59,1.29);
\draw[black, thin] (8.3,2) -- (9.9,2);
\draw[black, thin] (5,5.29) -- (3.9,6.4);
\draw[black, thin] (5,5.29) -- (4.4,6.7);
\draw[black, thin] (5,5.29) -- (5,6.9);
\draw[black, thin] (5,0.71) -- (3.5,-0.1);
\draw[black, thin] (5,0.71) -- (4.1,-0.5);
\draw[black, thin] (5,0.71) -- (4.8,-0.8);
\draw[black, thin] (6.3,5.48) -- (5.6,7);
\draw[black, thin] (6.3,5.48) -- (6.3,7);
\draw[black, thin] (6.3,5.48) -- (6.9,6.9);
\draw[black, thin] (7,0.712) -- (6.5,-0.96);
\draw[black, thin] (7,0.712) -- (7.2,-0.8);
\draw[black, thin] (7,0.712) -- (7.8,-0.6);
\draw[black, thin] (7.5,5) -- (8.5,6.1);
\draw[black, thin] (7.5,5) -- (7.8,6.6);
\draw[black, thin] (7.5,5) -- (9.1,5.5);
\draw[black, thin] (3.63,3.8) -- (2.09,2.3);
\draw[black, thin] (3.63,3.8) -- (1.99,2.9);
\draw[black, thin] (3.63,3.8) -- (2,3.5);
\draw[black, thin] (3.63,3.8) -- (2.1,4.1);
\draw[black, thin] (3.63,3.8) -- (2.4,4.8);
\draw[black, thin] (3.63,3.8) -- (2.9,5.57);
\draw[black, thin] (3.63,3.8) -- (3.4,6.1);
\draw[black, thin] (8.3,4) -- (9.9,4);
\draw[black, thin] (8.3,4) -- (10,3.3);
\draw[black, thin] (8.3,4) -- (9.7,4.6);
\filldraw [black] (3,0.34) circle (0.5pt);
\filldraw [black] (2.6,0.9) circle (0.5pt);
\filldraw [black] (2.3,1.5) circle (0.5pt);
\filldraw [black] (9.1,0.5) circle (0.5pt);
\filldraw [black] (9.59,1.29) circle (0.5pt);
\filldraw [black] (9.9,2) circle (0.5pt);
\filldraw [black] (3.9,6.4) circle (0.5pt);
\filldraw [black] (4.4,6.7) circle (0.5pt);
\filldraw [black] (5,6.9) circle (0.5pt);
\filldraw [black] (3.5,-0.1) circle (0.5pt);
\filldraw [black] (4.1,-0.5) circle (0.5pt);
\filldraw [black] (4.8,-0.8) circle (0.5pt);
\filldraw [black] (5.6,7) circle (0.5pt);
\filldraw [black] (6.3,7) circle (0.5pt);
\filldraw [black] (6.9,6.9) circle (0.5pt);
\filldraw [black] (6.5,-0.96) circle (0.5pt);
\filldraw [black] (7.2,-0.8) circle (0.5pt);
\filldraw [black] (7.8,-0.6) circle (0.5pt);
\filldraw [black] (8.5,6.1) circle (0.5pt);
\filldraw [black] (7.8,6.6) circle (0.5pt);
\filldraw [black] (9.1,5.5) circle (0.5pt);
\filldraw [black] (2.09,2.3) circle (0.5pt);
\filldraw [black] (1.99,2.9) circle (0.5pt);
\filldraw [black] (2,3.5) circle (0.5pt);
\filldraw [black] (2.1,4.1) circle (0.5pt);
\filldraw [black] (2.4,4.8) circle (0.5pt);
\filldraw [black] (2.9,5.57) circle (0.5pt);
\filldraw [black] (3.4,6.1) circle (0.5pt);
\filldraw [black] (9.9,4) circle (0.5pt);
\filldraw [black] (10,3.3) circle (0.5pt);
\filldraw [black] (9.7,4.6) circle (0.5pt);
\filldraw [black] (5,5.29) node[anchor=north] {$1$};
\filldraw [black] (6.3,5.48) node[anchor=north] {$6$};
\filldraw [black] (7.5,5) node[anchor=north] {$2$};
\filldraw [black] (8.3,4) node[anchor=north] {$7$};
\filldraw [black] (8.2,2.4) node[anchor=north] {$3$};
\filldraw [black] (7,1.2) node[anchor=north] {$8$};
\filldraw [black] (5,1.2) node[anchor=north] {$4$};
\filldraw [black] (3.9,2.3) node[anchor=north] {$9$};
\filldraw [black] (3.7,3.9) node[anchor=north] {$5$};
\filldraw [black] (3.9,6.9)  node[anchor=north] {$36$};
\filldraw [black] (4.4,7.2)node[anchor=north] {$28$};
\filldraw [black] (5,7.4) node[anchor=north] {$20$};
\filldraw [black] (4.3,6)  node[anchor=north] {{\tiny $49$}};
\filldraw [black] (4.4,6.5)node[anchor=north] {{\tiny $57$}};
\filldraw [black] (5.2,6.5) node[anchor=north] {{\tiny $65$}};
\filldraw [black] (5.6,7.5)  node[anchor=north] {$35$};
\filldraw [black] (6.3,7.5) node[anchor=north] {$27$};
\filldraw [black] (6.9,7.4) node[anchor=north] {$19$};
\filldraw [black] (5.7,6.7)  node[anchor=north] {{\tiny $45$}};
\filldraw [black] (6.16,6.8) node[anchor=north] {{\tiny $53$}};
\filldraw [black] (6.84,6.5) node[anchor=north] {{\tiny $61$}};
\filldraw [black] (8.7,6.6) node[anchor=north] {$26$};
\filldraw [black]  (7.8,7.1) node[anchor=north] {$34$};
\filldraw [black] (9.3,5.8) node[anchor=north] {$18$};
\filldraw [black] (8.4,6) node[anchor=north] {{\tiny $58$}};
\filldraw [black]  (7.5,6.3) node[anchor=north] {{\tiny $50$}};
\filldraw [black] (8.7,5.4) node[anchor=north] {{\tiny $66$}};
\filldraw [black] (10.2,4.2) node[anchor=north] {$25$};
\filldraw [black]  (10.3,3.6) node[anchor=north] {$17$};
\filldraw [black] (10,4.9)  node[anchor=north] {$33$};
\filldraw [black] (9.6,4.4) node[anchor=north] {{\tiny $54$}};
\filldraw [black]  (9.3,3.6) node[anchor=north] {{\tiny $62$}};
\filldraw [black] (9.1,4.8)  node[anchor=north] {{\tiny $46$}};
\filldraw [black] (9.2,0.5) node[anchor=north] {$16$};
\filldraw [black]  (9.9,1.5) node[anchor=north] {$24$};
\filldraw [black] (10.2,2.2)  node[anchor=north] {$32$};
\filldraw [black] (8.6,1) node[anchor=north] {{\tiny $67$}};
\filldraw [black]  (9.2,1.5) node[anchor=north] {{\tiny $59$}};
\filldraw [black] (9.5,2.5)  node[anchor=north] {{\tiny $51$}};
\filldraw [black] (6.5,-0.96) node[anchor=north] {$15$};
\filldraw [black]  (7.2,-0.8) node[anchor=north] {$23$};
\filldraw [black] (7.9,-0.6) node[anchor=north] {$31$};
\filldraw [black] (6.5,0) node[anchor=north] {{\tiny $63$}};
\filldraw [black]  (7,0) node[anchor=north] {{\tiny $55$}};
\filldraw [black] (7.5,0.5) node[anchor=north] {{\tiny $47$}};
\filldraw [black] (3.4,-0.1) node[anchor=north] {$14$};
\filldraw [black]  (4.1,-0.5) node[anchor=north] {$22$};
\filldraw [black] (4.8,-0.8) node[anchor=north] {$30$};
\filldraw [black] (4,0.7) node[anchor=north] {{\tiny $68$}};
\filldraw [black]  (4.6,0.1) node[anchor=north] {{\tiny $60$}};
\filldraw [black] (5.1,0.3) node[anchor=north] {{\tiny $52$}};
\filldraw [black] (2.8,0.5)  node[anchor=north] {$29$};
\filldraw [black] (2.3,1) node[anchor=north] {$21$};
\filldraw [black] (2,1.6) node[anchor=north] {$13$};
\filldraw [black] (3.5,1.3)  node[anchor=north] {{\tiny $48$}};
\filldraw [black] (2.8,1.6) node[anchor=north] {{\tiny $56$}};
\filldraw [black] (2.8,2.2) node[anchor=north] {{\tiny $64$}};
\filldraw [black] (1.8,2.5) node[anchor=north] {$10$};
\filldraw [black] (1.8,3.2) node[anchor=north] {$11$};
\filldraw [black]  (1.8,3.8) node[anchor=north] {$12$};
\filldraw [black] (1.8,4.4) node[anchor=north] {$37$};
\filldraw [black] (2.1,5.2)  node[anchor=north] {$38$};
\filldraw [black] (2.7,6) node[anchor=north] {$39$};
\filldraw [black] (3.17,6.5) node[anchor=north] {$40$};
\filldraw [black] (2.8,3) node[anchor=north] {{\tiny $71$}};
\filldraw [black] (2.4,3.5) node[anchor=north] {{\tiny $70$}};
\filldraw [black]  (2.4,4) node[anchor=north] {{\tiny $69$}};
\filldraw [black] (2.5,4.4) node[anchor=north] {{\tiny $44$}};
\filldraw [black] (2.7,5)  node[anchor=north] {{\tiny $43$}};
\filldraw [black] (3.2,5.5) node[anchor=north] {{\tiny $42$}};
\filldraw [black] (3.6,5.7) node[anchor=north] {{\tiny $41$}};
\filldraw [black] (4.3,4.8) node[anchor=north] {{\tiny $80$}};
\filldraw [black] (5.7,5.5) node[anchor=north] {{\tiny $79$}};
\filldraw [black] (7,5.3) node[anchor=north] {{\tiny $78$}};
\filldraw [black] (7.8,4.6) node[anchor=north] {{\tiny $77$}};
\filldraw [black] (3.8,3) node[anchor=north] {{\tiny $72$}};
\filldraw [black] (4.4,1.7) node[anchor=north] {{\tiny $73$}};
\filldraw [black] (6,1) node[anchor=north] {{\tiny $74$}};
\filldraw [black] (7.6,1.7) node[anchor=north] {{\tiny $75$}};
\filldraw [black] (8.2,3.1) node[anchor=north] {{\tiny $76$}};
\end{tikzpicture}
\caption{The graph $G_{9,3,4}$}
\label{fig1}
\end{figure}


\section{Unicyclic graph $G_{n,k,-c}$}\label{sec:4}

Let us consider the unicyclic graph $G_{n,k,-c}:= G(n;k,\dots ,k,k-c)$, $1\leq c\leq k$. For example, see Figure \ref{fig2a}. For $G_{n,k,-c}$, the number of vertices and edges are $p=q=n(k+1)-c$. 

Let $V(G_{n,k,-c})=V(C_{n})\cup \big\{a_{i,j}\colon 1\leq i\leq n-1, 1\leq j\leq k\big\}\cup \big\{a_{n,j}\colon 1\leq j\leq k-c\big\} $ and let the edge set be equal to  $$E(C_{n})\cup\big\{a_{i}a_{i,j}\colon 1\leq i\leq n-1,1\leq j\leq k\big\}\cup\big\{a_{n}a_{n,j}\colon 1\leq j \leq k-c\big\},$$
where $1\leq c\leq k$. 
\begin{figure}[ht]
\centering
\begin{tikzpicture}[scale=0.9]
\filldraw[color=black!100, fill=white!5,  thin](6,3) circle (2.5);
\filldraw [black] (5,5.29) circle (1pt);
\filldraw [black] (7,0.712) circle (1pt);
\filldraw [black] (3.63,3.8) circle (1pt);
\filldraw [black] (8.3,4) circle (1pt);
\filldraw [black] (4,1.5) circle (1pt);
\filldraw [black] (8.5,-0.1) circle (0.5pt); 
\draw[black, thin] (5,5.29) -- (3.9,6.4);
\draw[black, thin] (5,5.29) -- (5,6.9);
\draw[black, thin] (5,5.29) -- (6.3,7);
\draw[black, thin] (5,5.29) -- (2.9,5.57);
\draw[black, thin] (8.3,4) -- (10,3.3);
\draw[black, thin] (8.3,4) -- (9.6,4.8);
\filldraw [black] (9.6,4.8) circle (0.5pt); 
\draw[black, thin] (8.3,4) -- (8.5,6.1);
\draw[black, thin] (8.3,4) -- (9.9,2);
\draw[black, thin] (7,0.712) -- (6.5,-0.96);
\draw[black, thin] (7,0.712) -- (7.2,-0.8);
\draw[black, thin] (7,0.712) -- (9.1,0.5);
\draw[black, thin] (7,0.712) -- (8.5,-0.1);
\draw[black, thin] (4,1.5)  -- (4.1,-0.5);
\draw[black, thin] (4,1.5) -- (3,0.34);
\draw[black, thin] (4,1.5) -- (2.3,1.5);
\draw[black, thin] (4,1.5) -- (2.09,2.3);
\draw[black, thin] (3.63,3.8) -- (2,3.5);
\draw[black, thin] (3.63,3.8) -- (2.4,4.8);
\filldraw [black] (2.7,6) node[anchor=north] {$11$};
\filldraw [black] (3.8,5.9) node[anchor=north] {{\tiny $38$}};
\filldraw [black] (3.9,6.9)  node[anchor=north] {$15$};
\filldraw [black] (4.6,6.3) node[anchor=north] {{\tiny $34$}};
\filldraw [black] (5,7.4) node[anchor=north] {$19$};
\filldraw [black] (5.24,6.5) node[anchor=north] {{\tiny $30$}};
\filldraw [black] (6.3,7.5) node[anchor=north] {$23$};
\filldraw [black] (6.2,6.6) node[anchor=north] {{\tiny $26$}};
\filldraw [black] (8.7,6.6) node[anchor=north] {$10$};
\filldraw [black] (8.2,5.7) node[anchor=north] {{\tiny $36$}};
\filldraw [black] (9.9,5.1)   node[anchor=north] {$14$};
\filldraw [black] (9.1,4.5) node[anchor=north] {{\tiny $32$}};
\filldraw [black]  (10.3,3.6) node[anchor=north] {$18$};
\filldraw [black] (9.4,3.6) node[anchor=north] {{\tiny $28$}};
\filldraw [black] (10.2,2.2)  node[anchor=north] {$22$};
\filldraw [black] (9.4,2.5) node[anchor=north] {{\tiny $24$}};
\filldraw [black] (9.2,0.5) node[anchor=north] {$9$};
\filldraw [black] (8.6,0.98) node[anchor=north] {{\tiny $39$}};
\filldraw [black] (8.5,-0.1) node[anchor=north] {$13$};
\filldraw [black] (8.3,0.5) node[anchor=north] {{\tiny $35$}};
\filldraw [black]  (7.2,-0.8) node[anchor=north] {$17$};
\filldraw [black] (7.4,0) node[anchor=north] {{\tiny $31$}};
\filldraw [black] (6.5,-0.96) node[anchor=north] {$21$};
\filldraw [black] (6.4,-0.2) node[anchor=north] {{\tiny $27$}};
\filldraw [black]  (4.1,-0.5) node[anchor=north] {$8$};
\filldraw [black] (4.3,0.3) node[anchor=north] {{\tiny $37$}};
\filldraw [black] (2.8,0.5)  node[anchor=north] {$12$};
\filldraw [black] (3.5,0.8) node[anchor=north] {{\tiny $33$}};
\filldraw [black] (2,1.6) node[anchor=north] {$16$};
\filldraw [black] (2.9,1.6) node[anchor=north] {{\tiny $29$}};
\filldraw [black] (1.8,2.6) node[anchor=north] {$20$};
\filldraw [black] (2.6,2.6) node[anchor=north] {{\tiny $25$}};
\filldraw [black]  (1.8,3.8) node[anchor=north] {$6$};
\filldraw [black] (2.6,4.1) node[anchor=north] {{\tiny $41$}};
\filldraw [black] (2.2,5.1)  node[anchor=north] {$7$};
\filldraw [black] (2.9,4.9) node[anchor=north] {{\tiny $40$}};
\filldraw [black] (3,0.34) circle (0.5pt);
\filldraw [black] (2.3,1.5) circle (0.5pt);
\filldraw [black] (4.1,-0.5) circle (0.5pt);
\filldraw [black] (9.1,0.5) circle (0.5pt);
\filldraw [black] (9.9,2) circle (0.5pt);
\filldraw [black] (8.5,6.1) circle (0.5pt);
\filldraw [black] (3.9,6.4) circle (0.5pt);
\filldraw [black] (5,6.9) circle (0.5pt);
\filldraw [black] (6.3,7) circle (0.5pt);
\filldraw [black] (6.5,-0.96) circle (0.5pt);
\filldraw [black] (7.2,-0.8) circle (0.5pt);
\filldraw [black] (2.09,2.3) circle (0.5pt);
\filldraw [black] (2,3.5) circle (0.5pt);
\filldraw [black] (2.4,4.8) circle (0.5pt);
\filldraw [black] (2.9,5.57) circle (0.5pt);
\filldraw [black] (10,3.3) circle (0.5pt);
\filldraw [black] (5,5.29) node[anchor=north] {$1$};
\filldraw [black] (8.3,4) node[anchor=north] {$4$};
\filldraw [black] (7,1.2) node[anchor=north] {$2$};
\filldraw [black] (4.2,1.8) node[anchor=north] {$5$};
\filldraw [black] (3.8,3.9) node[anchor=north] {$3$};
\filldraw [black] (7,5.3) node[anchor=north] {{\tiny  $45$}};
\filldraw [black] (8.1,2.4) node[anchor=north] {{\tiny $44$}};
\filldraw [black] (6,1) node[anchor=north] {{\tiny $43$}};
\filldraw [black] (3.8,3) node[anchor=north] {{\tiny $42$}};
\filldraw [black] (4.3,4.8) node[anchor=north] {{\tiny $46$}};
\end{tikzpicture}
\caption{The graph $G_{5,4,-2}$}
\label{fig2a}
\end{figure}

\begin{theorem}\label{thm:k-c}
The unicyclic graph $G_{n,k,-c}$ is a super edge-magic total graph with super edge-magic total strength $$s m(G_{n,k,-c})=2n(k+1)-2c+\frac{n+3}{2}.$$
\end{theorem}

\begin{proof}
By Theorem \ref{thm:uni}, the graph $G_{n,k,-c}$ is super edge-magic total and the lower bound of its super edge-magic total strength is:
\begin{equation*}
\begin{split}
s m(G_{n,k,-c}) & \geq 2q+2+\frac{1}{q}\bigg(\frac{k(n-1)}{2} +\frac{n(n-1)}{2} -c(n-1) \bigg)\\
& = 2n(k+1)-2c+2+\frac{1}{n(k+1)-c}\bigg(\frac{n(n-1)(k+1)}{2} -c(n-1) \bigg)\\
& =  2n(k+1)-2c+2+\frac{n-1}{2}\bigg(\frac{n(k+1)-2c}{n(k+1)-c}\bigg).
\end{split}
\end{equation*}

Since $s m(G_{n,k,-c})$ is an integer, we consider the integer part of $\frac{n-1}{2}\bigg(\frac{n(k+1)-2c}{n(k+1)-c}\bigg)$. We have $\frac{n-1}{2}-\frac{n-1}{2}\bigg(\frac{n(k+1)-2c}{n(k+1)-c}\bigg)=\frac{(n-1)c}{2(n(k+1)-c)}$. Since $c\leq k$, we observe that $(n-1)c<2(n(k+1)-c)$. Therefore,

\begin{equation*}
\begin{split}
& 0< \frac{n-1}{2}-\frac{n-1}{2}\bigg(\frac{n(k+1)-2c}{n(k+1)-c}\bigg) < 1.
\end{split}
\end{equation*}

Hence for $G_{n,k,-c}$, we have
\begin{align*}  
s m(G_{n,k,-c})&\geq  2n(k+1)-2c+2+\frac{n-1}{2}\\
& = 2n(k+1)-2c+\frac{n+3}{2}.
\end{align*}
That is, 
\begin{equation}\label{eq:lb_k-c}
s m(G_{n,k,-c})\geq 2n(k+1)-2c+\frac{n+3}{2}.
\end{equation}
Now, we define a vertex labeling $f\colon V(G_{n,k,-c})  \longrightarrow \{1,\dots , p\}$ as follows:

For $1\leq i \leq n,$
\begin{equation}\label{eq:semG}
\begin{split}
&f(a_{i})= 
  \begin{cases} 
   \displaystyle{ \frac{i+1}{2}}& \text{ if } i \text{ is odd},  \\ \\
   \displaystyle{\frac{n+i+1}{2}}&  \text{ if } i \text{ is even}.
  \end{cases}\\
& f(a_{i,j})= n(k+2)-c-(n-1)j-i,\ 1\leq i\leq n-1,\ 1\leq j\leq k.\\
& f(a_{n,j})= n+j,\ 1\leq j\leq k-c. 
\end{split}
\end{equation}
As per the labeling defined in \eqref{eq:semG}, for any $u v\in E(G_{n,k,-c})$ we observe the following.
\begin{itemize}
\item If $u,v\in V(C_{n})$ then, $\{f(u)+f(v)\}=\big\{1+\frac{n+1}{2}, \dots , n+\frac{n+1}{2}\big\}$ is a consecutive sequence.
\item If $u=a_{n}$ and $v=a_{n,j}$, $1\leq j \leq k-c$, then we have  $\{f(u)+f(v)\}=\big\{n+\frac{n+3}{2} , \dots , n+k-c +\frac{n+1}{2} \big\}$, a consecutive sequence.
\item If $u=a_{i}$ and $v=a_{i,j}$, for $1\leq i\leq n-1,\ 1\leq j\leq k,$ then we have $\{f(u)+f(v)\}=\big\{n+k-c+\frac{n+3}{2} , \dots , n(k+1)-c+2 \big\}$, which is a consecutive sequence.
\end{itemize}

Thus we observe that $\{f(u)+f(v)\colon u v\in E(G_{n,k,-c})\}$ is a consecutive sequence with $\min \{f(u)+f(v)\colon u v\in E(G_{n,k,-c})\}= \frac{n+3}{2}$. Therefore by Lemma \ref{lem:SEM}, the vertex labeling $f$ extends to a super edge-magic total labeling of $G_{n,k,-c}$ with a magic constant $c(f)= 2n(k+1)-2c+\frac{n+3}{2}$. Hence,

\begin{equation}\label{eq:ub_k-c}
s m(G_{n,k,-c})\leq 2n(k+1)-2c+\frac{n+3}{2}.
\end{equation}

From \eqref{eq:lb_k-c} and \eqref{eq:ub_k-c}, we have
\begin{equation*}
 s m(G_{n,k,-c}) = 2n(k+1)-2c+\frac{n+3}{2}.
\end{equation*}

\end{proof}

\begin{example}
Super edge-magic total labeling of the graph $G_{5,4,-2}$ with super edge-magic total strength $s m(G_{5,4,-2})= 50$, is illustrated in Figure \ref{fig2a}.
\end{example}

\begin{example}
Super edge-magic total labeling of the graph $G_{5,8,-6}$ with super edge-magic total strength $s m(G_{5,8,-6})= 82$, is illustrated in Figure \ref{fig2}.
\end{example}

\begin{figure}[ht]
\centering
\begin{tikzpicture}
\filldraw[color=black!100, fill=white!5,  thin](6,3) circle (2.5);
\filldraw [black] (5,5.29) circle (1pt);
\filldraw [black] (7,0.712) circle (1pt);
\filldraw [black] (3.63,3.8) circle (1pt);
\filldraw [black] (8.3,4) circle (1pt);
\filldraw [black] (4,1.5) circle (1pt);
\filldraw [black] (5.6,-0.97) circle (0.5pt);
\filldraw [black] (8.5,-0.1) circle (0.5pt);
\filldraw [black] (2.09,1.9) circle (0.5pt);
\draw[black, thin] (5,5.29) -- (3.9,6.4);
\draw[black, thin] (5,5.29) -- (4.4,6.7);
\draw[black, thin] (5,5.29) -- (5,6.9);
\draw[black, thin] (5,5.29) -- (5.6,7);
\draw[black, thin] (5,5.29) -- (6.3,7);
\draw[black, thin] (5,5.29) -- (6.9,6.9);
\draw[black, thin] (5,5.29) -- (2.9,5.57);
\draw[black, thin] (5,5.29) -- (3.4,6.1);
\draw[black, thin] (8.3,4) -- (9.9,4);
\draw[black, thin] (8.3,4) -- (10,3.3);
\draw[black, thin] (8.3,4) -- (9.7,4.6);
\draw[black, thin] (8.3,4) -- (8.5,6.1);
\draw[black, thin] (8.3,4) -- (7.8,6.6);
\draw[black, thin] (8.3,4) -- (9.1,5.5);
\draw[black, thin] (8.3,4) -- (9.59,1.29);
\draw[black, thin] (8.3,4) -- (9.9,2);
\draw[black, thin] (7,0.712) -- (6.5,-0.96);
\draw[black, thin] (7,0.712) -- (7.2,-0.8);
\draw[black, thin] (7,0.712) -- (7.8,-0.6);
\draw[black, thin] (7,0.712) -- (5.6,-0.97);
\draw[black, thin] (7,0.712) -- (4.1,-0.5);
\draw[black, thin] (7,0.712) -- (4.8,-0.8);
\draw[black, thin] (7,0.712) -- (9.1,0.5);
\draw[black, thin] (7,0.712) -- (8.5,-0.1);
\draw[black, thin] (4,1.5) -- (3.5,-0.1);
\draw[black, thin] (4,1.5) -- (3,0.34);
\draw[black, thin] (4,1.5) -- (2.6,0.9);
\draw[black, thin] (4,1.5) -- (2.3,1.5);
\draw[black, thin] (4,1.5) -- (2.09,2.3);
\draw[black, thin] (4,1.5) -- (2.09,1.9);
\draw[black, thin] (4,1.5) -- (1.99,2.9);
\draw[black, thin] (4,1.5) -- (2,3.5);
\draw[black, thin] (3.63,3.8) -- (2.1,4.1);
\draw[black, thin] (3.63,3.8) -- (2.4,4.8);
\filldraw [black] (3,0.34) circle (0.5pt);
\filldraw [black] (2.6,0.9) circle (0.5pt);
\filldraw [black] (2.3,1.5) circle (0.5pt);
\filldraw [black] (9.1,0.5) circle (0.5pt);
\filldraw [black] (9.59,1.29) circle (0.5pt);
\filldraw [black] (9.9,2) circle (0.5pt);
\filldraw [black] (3.9,6.4) circle (0.5pt);
\filldraw [black] (4.4,6.7) circle (0.5pt);
\filldraw [black] (5,6.9) circle (0.5pt);
\filldraw [black] (3.5,-0.1) circle (0.5pt);
\filldraw [black] (4.1,-0.5) circle (0.5pt);
\filldraw [black] (4.8,-0.8) circle (0.5pt);
\filldraw [black] (5.6,7) circle (0.5pt);
\filldraw [black] (6.3,7) circle (0.5pt);
\filldraw [black] (6.9,6.9) circle (0.5pt);
\filldraw [black] (6.5,-0.96) circle (0.5pt);
\filldraw [black] (7.2,-0.8) circle (0.5pt);
\filldraw [black] (7.8,-0.6) circle (0.5pt);
\filldraw [black] (8.5,6.1) circle (0.5pt);
\filldraw [black] (7.8,6.6) circle (0.5pt);
\filldraw [black] (9.1,5.5) circle (0.5pt);
\filldraw [black] (2.09,2.3) circle (0.5pt);
\filldraw [black] (1.99,2.9) circle (0.5pt);
\filldraw [black] (2,3.5) circle (0.5pt);
\filldraw [black] (2.1,4.1) circle (0.5pt);
\filldraw [black] (2.4,4.8) circle (0.5pt);
\filldraw [black] (2.9,5.57) circle (0.5pt);
\filldraw [black] (3.4,6.1) circle (0.5pt);
\filldraw [black] (9.9,4) circle (0.5pt);
\filldraw [black] (10,3.3) circle (0.5pt);
\filldraw [black] (9.7,4.6) circle (0.5pt);
\filldraw [black] (5,5.29) node[anchor=north] {$1$};
\filldraw [black] (8.3,4) node[anchor=north] {$4$};
\filldraw [black] (7,1.2) node[anchor=north] {$2$};
\filldraw [black] (4.2,1.8) node[anchor=north] {$5$};
\filldraw [black] (3.8,3.9) node[anchor=north] {$3$};
\filldraw [black] (7,5.3) node[anchor=north] {{\tiny $77$}};
\filldraw [black] (8.1,2.4) node[anchor=north] {{\tiny $76$}};
\filldraw [black] (5.3,1.2) node[anchor=north] {{\tiny $75$}};
\filldraw [black] (3.8,3) node[anchor=north] {{\tiny $74$ }};
\filldraw [black] (4.4,4.8) node[anchor=north] {{\tiny $78$}};
\filldraw [black] (2.7,6) node[anchor=north] {$39$};
\filldraw [black] (3.5,5.6) node[anchor=north] {{\tiny $42$}};
\filldraw [black] (3.17,6.5) node[anchor=north] {$35$};
\filldraw [black] (3.7,6) node[anchor=north] {{\tiny $46$}};
\filldraw [black] (3.9,6.9)  node[anchor=north] {$31$};
\filldraw [black] (4,6.3) node[anchor=north] {{\tiny$50$}};
\filldraw [black] (4.4,7.2)node[anchor=north] {$27$};
\filldraw [black] (4.4,6.5) node[anchor=north] {{\tiny $54$}};
\filldraw [black] (5,7.4) node[anchor=north] {$23$};
\filldraw [black] (4.87,6.69) node[anchor=north] {{\tiny $58$}};
\filldraw [black] (5.6,7.5)  node[anchor=north] {$19$};
\filldraw [black] (5.3,6.74) node[anchor=north] {{\tiny $62$}};
\filldraw [black] (6.3,7.5) node[anchor=north] {$15$};
\filldraw [black] (5.8,6.8) node[anchor=north] {{\tiny $66$}};
\filldraw [black] (6.9,7.4) node[anchor=north] {$11$};
\filldraw [black] (6.3,6.78) node[anchor=north] {{\tiny $70$}};
\filldraw [black]  (7.8,7.1) node[anchor=north] {$38$};
\filldraw [black] (7.7,6.4) node[anchor=north] {{\tiny $40$}};
\filldraw [black] (8.7,6.6) node[anchor=north] {$34$};
\filldraw [black] (8.3,6) node[anchor=north] {{\tiny $44$}};
\filldraw [black] (9.4,5.8) node[anchor=north] {$30$};
\filldraw [black] (8.8,5.4) node[anchor=north] {{\tiny $48$}};
\filldraw [black] (10,4.9)  node[anchor=north] {$26$};
\filldraw [black] (9.2,4.78) node[anchor=north] {{\tiny $52$}};
\filldraw [black] (10.2,4.2) node[anchor=north] {$22$};
\filldraw [black] (9.4,4.34) node[anchor=north] {{\tiny $56$}};
\filldraw [black]  (10.3,3.6) node[anchor=north] {$18$};
\filldraw [black] (9.5,3.6) node[anchor=north] {{\tiny $60$}};
\filldraw [black] (10.2,2.2)  node[anchor=north] {$14$};
\filldraw [black] (9.5,3) node[anchor=north] {{\tiny $64$}};
\filldraw [black]  (9.9,1.5) node[anchor=north] {$10$};
\filldraw [black] (9.2,1.9) node[anchor=north] {{\tiny $68$}};
\filldraw [black] (9.2,0.5) node[anchor=north] {$37$};
\filldraw [black] (8.6,0.9) node[anchor=north] {{\tiny $43$}};
\filldraw [black] (8.5,-0.1) node[anchor=north] {$33$};
\filldraw [black] (8.2,0.4) node[anchor=north] {{\tiny $47$}};
\filldraw [black] (7.9,-0.6) node[anchor=north] {$29$};
\filldraw [black] (7.7,0.1) node[anchor=north] {{\tiny $51$}};
\filldraw [black]  (7.2,-0.8) node[anchor=north] {$25$};
\filldraw [black] (7.3,-0.1) node[anchor=north] {{\tiny $55$}};
\filldraw [black] (6.5,-0.96) node[anchor=north] {$21$};
\filldraw [black] (6.83,-0.2) node[anchor=north] {{\tiny $59$}};
\filldraw [black] (5.6,-0.97) node[anchor=north] {$17$};
\filldraw [black] (6.24,-0.25) node[anchor=north] {{\tiny $63$}};
\filldraw [black] (4.8,-0.8) node[anchor=north] {$13$};
\filldraw [black] (5.6,-0.2) node[anchor=north] {{\tiny $67$}};
\filldraw [black]  (4.1,-0.5) node[anchor=north] {$9$};
\filldraw [black] (4.86,-0.12) node[anchor=north] {{\tiny $71$}};
\filldraw [black] (3.4,-0.1) node[anchor=north] {$36$};
\filldraw [black] (3.8,0.6) node[anchor=north] {{\tiny $41$}};
\filldraw [black] (2.8,0.5)  node[anchor=north] {$32$};
\filldraw [black] (3.4,0.8) node[anchor=north] {{\tiny $45$}};
\filldraw [black] (2.3,1) node[anchor=north] {$28$};
\filldraw [black] (3.1,1.2) node[anchor=north] {{\tiny $49$}};
\filldraw [black] (2,1.6) node[anchor=north] {$24$};
\filldraw [black] (2.8,1.6) node[anchor=north] {{\tiny $53$}};
\filldraw [black] (1.8,2.2) node[anchor=north] {$20$};
\filldraw [black] (2.5,1.9) node[anchor=north] {{\tiny $57$}};
\filldraw [black] (1.8,2.6) node[anchor=north] {$16$};
\filldraw [black] (2.5,2.45) node[anchor=north] {{\tiny $61$}};
\filldraw [black] (1.8,3.2) node[anchor=north] {$12$};
\filldraw [black] (2.5,2.89) node[anchor=north] {{\tiny $65$}};
\filldraw [black]  (1.8,3.8) node[anchor=north] {$8$};
\filldraw [black] (2.5,3.5) node[anchor=north] {{\tiny $69$}};
\filldraw [black] (1.8,4.4) node[anchor=north] {$6$};
\filldraw [black] (2.5,4.1) node[anchor=north] {{\tiny $73$}};
\filldraw [black] (2.2,5.1)  node[anchor=north] {$7$};
\filldraw [black] (2.9,4.8) node[anchor=north] {{\tiny $72$}};
\end{tikzpicture}
\caption{The graph $G_{5,8,-6}$}
\label{fig2}
\end{figure}
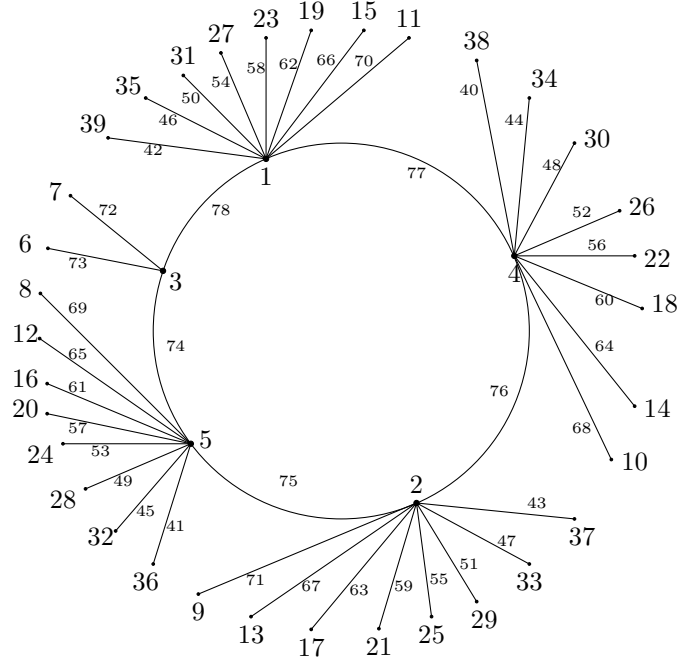

\section{Unicyclic graph $G(n;k,r)$}\label{sec:5}

Let $G(n;k,r)$ be the unicyclic graph $G(n,k_{1},\dots ,k_{n})$ with $k_{i}=k,$ if $i\neq r,n-r$ and $k_{r}=k_{n-r}=k+1$ for any odd number $r$, $1\leq r< n$. See Figure \ref{fig3a}.
Let $p=q=n(k+1)+2$, be the number of vertices and edges of $G(n;k,r)$.
Let $$V(G(n;k,r))=V(C_{n})\cup \big\{a_{i,j}\colon 1\leq i\leq n, 1\leq j\leq k\big\}\cup \{a_{r,k+1},a_{n-r,k+1}\},$$
and the edge set $E(G(n;k,r))$ be
\begin{equation*}
\begin{split}
& E(C_{n})\cup\big\{a_{i}a_{i,j}\colon 1\leq i\leq n,1\leq j\leq k\big\}\cup\big\{a_{i}a_{i,k+1}\colon i\in \{r,n-r\}\big\}.
\end{split}
\end{equation*}


For the unicyclic graph $G(n;k,r)$, we have the following theorem.
\begin{theorem}\label{thm:3}
The unicyclic graph $G(n;k,r)$, where $r$ is any odd number such that $1\leq r< n$, admits a super edge-magic total labeling and has a super edge-magic total strength $$sm(G(n;k,r))=2n(k+1)+4+\frac{n+3}{2}.$$
\end{theorem}

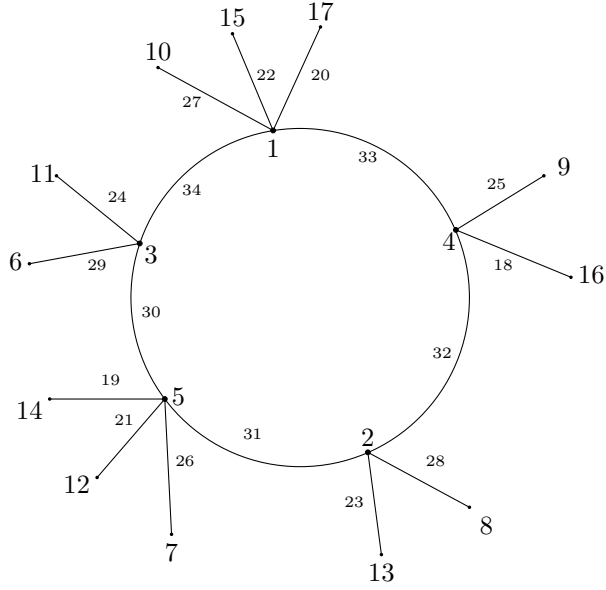
\begin{figure}[ht]
\centering
\begin{tikzpicture}[scale=0.9]
\filldraw[color=black!100, fill=white!5,  thin](6,3) circle (2.5);
\filldraw [black] (5.6,5.47) circle (1pt);
\filldraw [black] (7,0.712) circle (1pt);
\filldraw [black] (3.63,3.8) circle (1pt);
\filldraw [black] (8.3,4) circle (1pt);
\filldraw [black] (4,1.5) circle (1pt);
\filldraw [black] (8.5,-0.1) circle (0.5pt); 
\draw[black, thin] (5.6,5.47) -- (3.9,6.4);
\draw[black, thin] (5.6,5.47) -- (5,6.9);
\draw[black, thin] (5.6,5.47) -- (6.3,7);
\draw[black, thin] (8.3,4) -- (10,3.3);
\draw[black, thin] (8.3,4) -- (9.6,4.8);
\filldraw [black] (9.6,4.8) circle (0.5pt); 
\draw[black, thin] (7,0.712) -- (7.2,-0.8);
\draw[black, thin] (7,0.712) -- (8.5,-0.1);
\draw[black, thin] (4,1.5)  -- (4.1,-0.5);
\draw[black, thin] (4,1.5) -- (3,0.34);
\draw[black, thin] (4,1.5) -- (2.3,1.5);
\draw[black, thin] (3.63,3.8) -- (2,3.5);
\draw[black, thin] (3.63,3.8) -- (2.4,4.8);
\filldraw [black] (3.9,6.9)  node[anchor=north] {$10$};
\filldraw [black] (4.4,6.1) node[anchor=north] {{\tiny $27$}};
\filldraw [black] (5,7.4) node[anchor=north] {$15$};
\filldraw [black] (5.5,6.5) node[anchor=north] {{\tiny $22$}};
\filldraw [black] (6.3,7.5) node[anchor=north] {$17$};
\filldraw [black] (6.3,6.5) node[anchor=north] {{\tiny $20$}};
\filldraw [black] (9.9,5.2)   node[anchor=north] {$9$};
\filldraw [black] (8.9,4.9) node[anchor=north] {{\tiny $25$}};
\filldraw [black]  (10.3,3.6) node[anchor=north] {$16$};
\filldraw [black] (9,3.7) node[anchor=north] {{\tiny $18$}};
\filldraw [black] (8.75,-0.1) node[anchor=north] {$8$};
\filldraw [black] (8,0.8) node[anchor=north] {{\tiny $28$}};
\filldraw [black]  (7.2,-0.8) node[anchor=north] {$13$};
\filldraw [black] (6.8,0.2) node[anchor=north] {{\tiny $23$}};
\filldraw [black]  (4.1,-0.5) node[anchor=north] {$7$};
\filldraw [black] (4.3,0.8) node[anchor=north] {{\tiny $26$}};
\filldraw [black] (2.7,0.5)  node[anchor=north] {$12$};
\filldraw [black] (3.4,1.4) node[anchor=north] {{\tiny $21$}};
\filldraw [black] (2,1.6) node[anchor=north] {$14$};
\filldraw [black] (3.2,2) node[anchor=north] {{\tiny $19$}};
\filldraw [black]  (1.8,3.8) node[anchor=north] {$6$};
\filldraw [black] (3,3.7) node[anchor=north] {{\tiny $29$}};
\filldraw [black] (2.2,5.1)  node[anchor=north] {$11$};
\filldraw [black] (3.3,4.7) node[anchor=north] {{\tiny $24$}};
\filldraw [black] (3,0.34) circle (0.5pt);
\filldraw [black] (2.3,1.5) circle (0.5pt);
\filldraw [black] (4.1,-0.5) circle (0.5pt);
\filldraw [black] (3.9,6.4) circle (0.5pt);
\filldraw [black] (5,6.9) circle (0.5pt);
\filldraw [black] (6.3,7) circle (0.5pt);
\filldraw [black] (7.2,-0.8) circle (0.5pt);
\filldraw [black] (2,3.5) circle (0.5pt);
\filldraw [black] (2.4,4.8) circle (0.5pt);
\filldraw [black] (10,3.3) circle (0.5pt);
\filldraw [black] (5.6,5.47) node[anchor=north] {$1$};
\filldraw [black] (8.2,4.1) node[anchor=north] {$4$};
\filldraw [black] (7,1.2) node[anchor=north] {$2$};
\filldraw [black] (4.2,1.8) node[anchor=north] {$5$};
\filldraw [black] (3.8,3.9) node[anchor=north] {$3$};
\filldraw [black] (7,5.3) node[anchor=north] {{\tiny $33$}};
\filldraw [black] (8.1,2.4) node[anchor=north] {{\tiny  $32$}};
\filldraw [black] (5.3,1.2) node[anchor=north] {{\tiny $31$}};
\filldraw [black] (3.8,3) node[anchor=north] { {\tiny $30$}};
\filldraw [black] (4.4,4.8) node[anchor=north] {{\tiny $34$}};
\end{tikzpicture}
\caption{The graph $G(5;2,1)$}
\label{fig3a}
\end{figure}

\begin{proof}
By Theorem \ref{thm:uni}, the unicyclic graph $G(n;k,r)$ is a super edge-magic total graph with

\begin{equation*}
\begin{split}
sm(G(n;k,r)) & \geq 2q+2+\frac{1}{q}\bigg((k+1)+2k+\dots + (n-1)k+\frac{n(n-1)}{2}\bigg) \\
& = 2n(k+1)+6+\frac{1}{n(k+1)+2}\bigg(\frac{nk(n-1)}{2} +\frac{n(n-1)}{2} +1\bigg)\\
& =  2n(k+1)+6+\frac{n-1}{2}\bigg(\frac{n(k+1)}{n(k+1)+2}\bigg)+ \frac{1}{n(k+1)+2}.
\end{split}
\end{equation*}
That is, we have $sm(G(n;k,r)) \geq 2n(k+1)+6+\frac{n-1}{2}\big(\frac{n(k+1)}{n(k+1)+2}\big)+ \frac{1}{n(k+1)+2}. $
We know that $s m(G(n;k,r))$ is an integer and we can observe that
\begin{equation*}
\begin{split}
& \frac{n-1}{2} - \bigg (\frac{n-1}{2}\bigg(\frac{n(k+1)}{n(k+1)+2}\bigg) + \frac{1}{n(k+1)+2} \bigg)\\
& = \frac{n-1}{2} \bigg(\frac{2}{n(k+1)+2}\bigg)- \frac{1}{n(k+1)+2} \\
& = \frac{n-2}{n(k+1)+2}\\  
& < 1. 
\end{split}
\end{equation*}
Hence, we observe that the integer part of $\frac{n-1}{2}\big(\frac{n(k+1)}{n(k+1)+2}\big)+ \frac{1}{n(k+1)+2}$ is $\frac{n-1}{2}$. And we have
\begin{align*}  
sm(G(n;k,r)) & \geq 2n(k+1)+6+\frac{n-1}{2} = 2n(k+1)+4+\frac{n+3}{2}.
\end{align*}

Therefore, we can express
\begin{equation}\label{eq:3}
sm(G(n;k,r)) \geq  2n(k+1)+4+\frac{n+3}{2}.
\end{equation}

Now, if we prove that there exists a super edge-magic total labeling $f$ of $G(n;k,r)$ with magic constant $c(f)=2n(k+1)+4+\frac{n+3}{2}$, then our proof is complete. 

Let us define a vertex labeling $f\colon V(G(n;k,r))  \longrightarrow \{1,\dots , p\}$ as follows.
For $1\leq i \leq n,$ 
\begin{equation}\label{eq:semlabel1}
f(a_{i})= 
  \begin{cases} 
   \displaystyle{ \frac{i+1}{2}}& \text{ if } i \text{ is odd},  \\ \\
   \displaystyle{\frac{n+i+1}{2}}&  \text{ if } i \text{ is even}.
  \end{cases}
\end{equation}
And,

\begin{equation}\label{eq:semlabel2}
\begin{split}
& f(a_{i,j})= n(k-j+2)-2f(a_{i})+2 ,\ 1\leq i\leq n,\ 1\leq j\leq k-1,\\
&  f(a_{i,k})= 
  \begin{cases} 
   2n +2-2f(a_{i}) & \text{ if }  1\leq f(a_{i})\leq \frac{n+1}{2},  \\ 
   n(k+2) +4 - 2f(a_{i})& \text{ if }  \frac{n+3}{2}\leq f(a_{i})\leq n-f(a_{r}),   \\
   n(k+2) + 2 - 2f(a_{i})& \text{ if }  n-f(a_{r})+1 \leq f(a_{i})\leq n,
  \end{cases}\\
& f(a_{r,k+1})= n(k+1) +2 ,\\
& f(a_{n-r,k+1})= n k+ r +3.
\end{split}
\end{equation}

As per the above labeling, for $u v\in E(G(n;k,r))$ we observe that:

\begin{itemize}
\item For $u,v\in V(C_{n})$, $\{f(u)+f(v)\}=\big\{1+\frac{n+1}{2}, \dots , n+\frac{n+1}{2}\big\}$ is a consecutive sequence.
\item Let us consider $u=a_{i}$ and $v=a_{i,j}$, for any  $1\leq i\leq n,\ 1\leq j\leq k-1$. Then  $\{f(u)+f(v)\}=\big\{n(k-j+2)-f(a_{i})+2 \colon 1\leq i\leq n,\ 1\leq j\leq k-1\big\}$ is a consecutive sequence with minimal element $2n+2$ and maximal element $n (k+1)+1$.
\item Let $u=a_{i}$ and $v=a_{i,k}$, $1\leq i\leq n$. 
 \begin{itemize}
     \item If $1\leq f(a_{i})\leq \frac{n+1}{2},$ then $\{f(a_{i})+f(a_{i,k})\}= \big\{n+\frac{n+3}{2},\dots , 2n+1\big\}$, is a consecutive sequence.
     \item If $\frac{n+3}{2}\leq f(a_{i})\leq n-f(a_{r})$, then $\{f(a_{i})+f(a_{i,k})\}$ is consecutive and equals $\bigg\{n(k+1)+4+f(a_{r}),\dots , n(k+1)+4+\frac{n-3}{2}\bigg\}.$
     \item If $n-f(a_{r})+1 \leq f(a_{i})\leq n,$ then we see that the set $\{f(a_{i})+f(a_{i,k})\}=\bigg\{n(k+1)+2,\dots , n(k+1)+1+f(a_{r})\bigg\}$, is consecutive. 
 \end{itemize}
\item For $u=a_{r}$ and $v=a_{r,k+1}$, $f(u)+f(v)= n (k+1)+2+f(a_{r})$.
\item If $u=a_{n-r}$ and $v=a_{n-r,k+1}$, then $f(u)+f(v)= n (k+1)+3 +f(a_{r})$.
\end{itemize}

Therefore, we observe that  $\{f(u)+f(v)\colon u v\in E(G(n;k,r))\}$ is a consecutive sequence whose minimum element is $\frac{n+3}{2}$. Hence by Lemma \ref{lem:SEM}, the vertex labeling $f$ extends to a super edge-magic total labeling of $G(n;k,r)$ with magic constant $c(f)= 2n(k+1)+4+\frac{n+3}{2}$. Hence,

\begin{equation}\label{eq:5}
s m(G(n;k,r))\leq 2n(k+1)+4+\frac{n+3}{2}.
\end{equation}

From \eqref{eq:3} and \eqref{eq:5}, we have

\begin{equation*}
\begin{split}
& 2n(k+1)+4+\frac{n+3}{2} \leq s m(G(n;k,r))\leq 2n(k+1)+4+\frac{n+3}{2} \\
&\implies s m(G(n;k,r)) = 2n(k+1)+4+\frac{n+3}{2}.
\end{split}
\end{equation*}

\end{proof}


\begin{example}
Super edge-magic total labeling of the graph $G(5;2,1)$ with super edge-magic total strength $s m(G(5;2,1))= 38$ is illustrated in Figure \ref{fig3a}. 
\end{example}

\begin{example}
Super edge-magic total labeling of the graph $G(5;2,3)$ with super edge-magic total strength $s m(G(5;2,3))= 38$ is illustrated in Figure \ref{fig3b}. 
\end{example}

\begin{figure}[ht]
\centering
\begin{tikzpicture}[scale=0.9]
\filldraw[color=black!100, fill=white!5,  thin](6,3) circle (2.5);
\filldraw [black] (5.6,5.47) circle (1pt);
\filldraw [black] (7,0.712) circle (1pt);
\filldraw [black] (3.63,3.8) circle (1pt);
\filldraw [black] (8.3,4) circle (1pt);
\filldraw [black] (4,1.5) circle (1pt);
\filldraw [black] (8.5,-0.1) circle (0.5pt); 
\draw[black, thin] (5.6,5.47) -- (3.9,6.4);
\draw[black, thin] (5.6,5.47) -- (6.3,7);
\draw[black, thin] (8.3,4) -- (10,3.3);
\draw[black, thin] (8.3,4) -- (9.6,4.8);
\filldraw [black] (9.6,4.8) circle (0.5pt); 
\draw[black, thin] (8.3,4) -- (8.5,6.1);
\filldraw [black] (8.5,6.1) circle (0.5pt); 
\draw[black, thin] (7,0.712) -- (7.2,-0.8);
\draw[black, thin] (7,0.712) -- (9.1,0.5);
\filldraw [black] (9.1,0.5) circle (0.5pt); 
\draw[black, thin] (7,0.712) -- (8.5,-0.1);
\draw[black, thin] (4,1.5)  -- (4.1,-0.5);
\draw[black, thin] (4,1.5) -- (2.3,1.5);
\draw[black, thin] (3.63,3.8) -- (2,3.5);
\draw[black, thin] (3.63,3.8) -- (2.4,4.8);
\filldraw [black] (4.4,6.1) node[anchor=north] {{\tiny $27$}};
\filldraw [black] (6.3,6.5) node[anchor=north] {{\tiny $22$}};
\filldraw [black] (8.9,4.9) node[anchor=north] {{\tiny $20$}};
\filldraw [black] (8.2,5.5) node[anchor=north] {{\tiny $25$}};
\filldraw [black] (9,3.7) node[anchor=north] {{\tiny $18$}};
\filldraw [black] (8.2,1.1) node[anchor=north] {{\tiny $28$}};
\filldraw [black] (7.6,0.4) node[anchor=north] {{\tiny $23$}};
\filldraw [black] (6.8,0.2) node[anchor=north] {{\tiny $19$}};
\filldraw [black] (4.3,0.8) node[anchor=north] {{\tiny $26$}};
\filldraw [black] (3.2,2) node[anchor=north] {{\tiny $21$}};
\filldraw [black] (3,3.7) node[anchor=north] {{\tiny $29$}};
\filldraw [black] (3.3,4.7) node[anchor=north] {{\tiny $24$}};
\filldraw [black] (3.9,6.9)  node[anchor=north] {$10$};
\filldraw [black] (6.3,7.5) node[anchor=north] {$15$};
\filldraw [black] (8.6,6.6) node[anchor=north] {$9$};
\filldraw [black] (10,5.2)   node[anchor=north] {$14$};
\filldraw [black]  (10.3,3.6) node[anchor=north] {$16$};
\filldraw [black] (9.4,0.8) node[anchor=north] {$8$};
\filldraw [black] (8.5,-0.1) node[anchor=north] {$13$};
\filldraw [black]  (7.2,-0.8) node[anchor=north] {$17$};
\filldraw [black]  (4.1,-0.5) node[anchor=north] {$7$};
\filldraw [black] (2,1.8) node[anchor=north] {$12$};
\filldraw [black]  (1.8,3.8) node[anchor=north] {$6$};
\filldraw [black] (2.2,5.1)  node[anchor=north] {$11$};
\filldraw [black] (2.3,1.5) circle (0.5pt);
\filldraw [black] (4.1,-0.5) circle (0.5pt);
\filldraw [black] (3.9,6.4) circle (0.5pt);
\filldraw [black] (6.3,7) circle (0.5pt);
\filldraw [black] (7.2,-0.8) circle (0.5pt);
\filldraw [black] (2,3.5) circle (0.5pt);
\filldraw [black] (10,3.3) circle (0.5pt);
\filldraw [black] (5.6,5.47) node[anchor=north] {$1$};
\filldraw [black] (8.2,4) node[anchor=north] {$4$};
\filldraw [black] (7,1.2) node[anchor=north] {$2$};
\filldraw [black] (4.2,1.8) node[anchor=north] {$5$};
\filldraw [black] (3.8,3.9) node[anchor=north] {$3$};
\filldraw [black] (7,5.3) node[anchor=north] {{\tiny $33$}};
\filldraw [black] (8.1,2.4) node[anchor=north] {{\tiny  $32$}};
\filldraw [black] (5.3,1.2) node[anchor=north] {{\tiny $31$}};
\filldraw [black] (3.8,3) node[anchor=north] { {\tiny $30$}};
\filldraw [black] (4.4,4.8) node[anchor=north] {{\tiny $34$}};
\end{tikzpicture}
\caption{The graph $G(5;2,3)$}
\label{fig3b}
\end{figure}
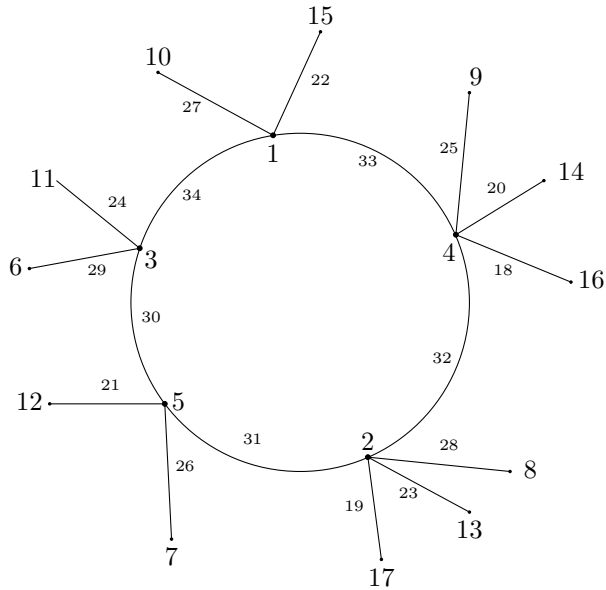

\section{Conclusions}\label{sec:6}

In this study, we determine the super edge-magic total strength of three variations of $G(n;k_{1},\dots ,k_{n})$, unicyclic $(p,q)$-graphs. All three of them have super edge-magic total strength equal to $2q + \frac{n+3}{2}.$ These results can be considered as the preliminary steps to provide evidence in proving the Conjecture \ref{cnj:main}.


\end{document}